\documentclass[11pt]{article}

\usepackage[pdftex]{graphicx}

\usepackage{amssymb}
\usepackage{amsthm}

\usepackage{accents}

\usepackage{tipa}

\makeatletter
\renewcommand\section{\@startsection{section}{1}{\z@}%
                                  {-3.5ex \@plus -1ex \@minus -.2ex}%
                                  {2.3ex \@plus.2ex}%
                                  {\normalfont\large\bfseries}}
\makeatother

\DeclareGraphicsExtensions{.jpg}
\begin{document}

\title{Sufficient condition for Reed's conjecture}

\author{Misa Nakanishi \thanks{E-mail address : nakanishi@2004.jukuin.keio.ac.jp}}
\date{}
\maketitle

\begin{abstract}
From the research of several
recent papers, we are concerned with domination
number in cubic graphs and give a sufficient condition for Reed's conjecture. \\
keywords: cubic graph, minimum dominating set, independent dominating set
\end{abstract}

\newtheorem{thm}{Theorem}[section]
\newtheorem{lem}{Lemma}[section]
\newtheorem{prop}{Proposition}[section]
\newtheorem{cor}{Corollary}[section]
\newtheorem{rem}{Remark}[section]
\newtheorem{conj}{Conjecture}[section]
\newtheorem{claim}{Claim}[section]
\newtheorem{obs}{Observation}[section]
\newtheorem{fact}{Fact}[section]

\newtheorem{defn}{Definition}[section]
\newtheorem{propa}{Proposition}
\renewcommand{\thepropa}{\Alph{propa}}
\newtheorem{conja}[propa]{Conjecture}

\section{Introduction}
\label{intro}

\noindent This study considers a graph $G$, which is finite, undirected, and simple, with the vertex set $V$ and edge set $E$. We follow the notations presented in \cite{Diestel}. For a vertex $v \in V(G)$, the open neighborhood, denoted by $N_G(v)$, is $\{ u \in V(G) \colon\ uv \in E(G) \}$, and the closed neighborhood, denoted by $N_G[v]$, is $N_G(v) \cup \{v\}$; in addition, for a set $W \subseteq V(G)$, let $N_G(W) = {\displaystyle \bigcup_{v \in W}} N_G(v)$ and $N_G[W] = N_G(W) \cup W$. A {\it dominating
set} $X \subseteq V(G)$ is such that $N_G[X] = V(G)$. A minimum dominating set is called a {\it d-set}. The minimum cardinality taken over all minimal dominating sets of $G$ is the {\it domination number} denoted by $\gamma(G)$. The minimum cardinality taken over all maximal independent sets of $G$ is the {\it independent domination number} denoted by $i(G)$. For a dominating set $X \subseteq V(G)$ and a set $R \subseteq V(G)$, $X(R)$ denotes $X \cap R$. For a set $S \subseteq V(G)$, as is clear from the context, $S$ denotes $G[S]$. \\

\noindent For the domination number of a graph, in decades, cubic graphs have been intensively studied and several important results were shown. The complexity of a minimum dominating set (MDS) in cubic graphs is NP-hard \cite{Alimonti}. A random 3-regular graph asymptotically almost surely has no 3-star factors \cite{Assiyatun}. Reed indicated that almost all cubic graphs are Hamiltonian. In addition, the upper bound of the domination number of a connected cubic graph $G$ was conjectured as being $\lceil |V(G)|/3 \rceil$ \cite{Reed}. Then, the counterexamples that exceed the bound demonstrate that, for example, an extremal graph of the domination number of 21 over 60 vertices exists, following the series of cubic graphs beyond the boundary \cite{Kostochka} \cite{Kelmans} \cite{Stodolsky}. \\

\noindent In this paper, we show that the connected cubic graphs that have the domination number above the bound have a minimum dominating set as an independent set. Otherwise, the conjecture is true. \\

\noindent A sufficient condition for a general graph $G$ to have $\gamma(G) = i(G)$ was represented as an induced subgraph isomorphic to $K_{1, 3}$, also called 3-star, free \cite{Allan}. Next, a double star such that both centers have degrees at least three, say $I$, was introduced. We observed $I$ as a forbidden subgraph for $\gamma(G) = i(G)$ with the simplest proof.

\begin{propa}[\cite{Allan}]
If a graph $G$ does not have an induced subgraph isomorphic to $K_{1, 3}$, then $\gamma(G) = i(G)$.
\end{propa}

\begin{propa}[\cite{Cockayne}]
For a graph $G$, if $I \not \subseteq G$, then $\gamma(G) = i(G)$. 
\end{propa}

\begin{proof}
Suppose that $X$ is a d-set of $G$ with $E(X)$ minimal and nonempty. For two vertices $x, y \in X$ such that $xy \in E(G)$, it follows that $d_G(x) \geq 2$ and $d_G(y) \geq 2$, for otherwise, contrary to the minimality of $X$. Suppose $d_{G}(x) = 2$ or $d_G(y) = 2$. It suffices that $d_G(x) = 2$. Set $N_{G}(x) \setminus \{y\} = \{x'\}$. For all $z \in N_{G}(x') \setminus \{x\}$, if $z \notin X$, then $(X \setminus \{x\}) \cup \{x'\} = X'$ that is a d-set of $G$. Now, $||X|| - 1 \geq ||X'||$, contrary to the minimality of $E(X)$. If there exists $z \in N_{G}(x') \setminus \{x\}$ such that $z \in X$, then $X \setminus \{x\} = X''$ that is a d-set of $G$ and contrary to the minimality of $X$. Thus it follows that $d_G(x) \geq 3$ and $d_G(y) \geq 3$, and so $I \subseteq G$.
\end{proof}

\noindent A 3-connected cubic graph was conjectured to be one for which the difference between the independent domination number and the domination number is one; however, it was not true.

\begin{propa}[\cite{Zverovich}]
For every $c \in \{0, 1, 2, 3\}$ and every integer $k$ such that $k \geq 0$, there exist infinitely many cubic graphs with connectivity $c$ (say one as $G$) for which $i(G) - \gamma(G) = k$.
\end{propa}

\noindent The next statement was suggested and has been widely discussed. 

\begin{conja}[\cite{Reed}]
Every connected cubic graph $G$ contains a dominating set of at most $\lceil |V(G)|/3 \rceil$ vertices.
\end{conja}

\noindent There are some counterexamples in cubic graphs with connectivity one and two. For example, we can observe on the graph $H_4$ in \cite{Kostochka}. \\

\section{Sufficient condition for Reed's conjecture}
\label{sec:3}
 
It is central for this proof how edges and vertices are deleted from a cubic graph to preserve its dominating set. By deleting a vertex of $G$, some path may be broken.
A substitution is needed to connect them and preserve an original set of vertices of the dominating set. 

\begin{lem}\label{disjoint}
Let a graph $G$ with $\Delta(G) \leq 3$ have a d-set $X$ with $|X| \geq 3$ and with $E(X)$ minimal and nonempty. Then for each $v_1, v_2, w \in X$ such that $v_1v_2 \in E(X)$, we have  $N_G[\{v_1, v_2\}] \cap N_G[w] = \emptyset$.
\end{lem}

\begin{proof}
For some $x_1, x_2 \in X$, let $x_1x_2 \in E(X)$. If $|N_G(x_1) \setminus \{x_2\}| = 0$, then $X \setminus \{x_1\}$ is a d-set of $G$, contrary to the minimality of $X$. Let $|N_G(x_1) \setminus \{x_2\}| = 1$ and set $N_G(x_1) \setminus \{x_2\} = \{v_1\}$. For some $v_2 \in N_G(v_1) \setminus \{x_1\}$, suppose $v_1 \in X$ or $v_2 \in X$. Now, $X \setminus \{x_1\}$ is a d-set of $G$, contrary to the minimality of $X$. For all $v_2 \in N_G(v_1) \setminus \{x_1\}$, suppose $v_1 \notin X$ and $v_2 \notin X$. Now, $(X \setminus \{x_1\}) \cup \{v_1\}$ is a d-set of $G$, contrary to the minimality of $E(X)$. Thus $|N_G(x_1) \setminus \{x_2\}| = 2$ and set $N_G(x_1) \setminus \{x_2\} = \{v_1, w_1\}$. For some $v_2 \in N_G(v_1) \setminus \{x_1\}$, suppose $v_1 \in X$ or $v_2 \in X$. Now, if $w_2 \in N_G(w_1) \setminus \{x_1\}$, then $w_1 \notin X$ and $w_2 \notin X$, for otherwise, $X \setminus \{x_1\}$ is a d-set of $G$, contrary to the minimality of $X$. Thus $(X \setminus \{x_1\}) \cup \{w_1\}$ is a d-set of $G$, contrary to the minimality of $E(X)$. Therefore, if $v_2 \in N_G(v_1) \setminus \{x_1\}$, then $v_1 \notin X$ and $v_2 \notin X$ as required.
\end{proof}

\begin{defn}
For a graph $G$, two vertices $v_1, v_2 \in V(G)$ such that $v_1v_2 \in E(G)$, and a set $X \subseteq V(G)$, suppose that (i), (ii), or (iii) holds. \\
(i) $v_1, v_2 \not \in X$ \\
(ii) $v_1 \in X$, $v_2 \notin X$, $(N_G(v_2) \setminus \{v_1\}) \cap X \ne \emptyset$ \\
(iii) $v_1, v_2 \in X$ \\
Then the set of all $v_1v_2$ is denoted by $U_G(X)$, or $U(X)$.
\end{defn}

\begin{fact}\label{U'}
For a graph $G$ and its d-set $X$, let $U' \subseteq U_G(X)$. Then $X$ is a d-set of $G - U'$. 
\end{fact}

\begin{defn}
For a graph $G$ and a set $Y \subseteq V(G)$, suppose that $t_1 \in Y$ has a set of vertices $B(t_1)$ such that for all $b \in B(t_1)$, $b \in N_G(t_1)$ and $(N_G[b] \setminus \{t_1\}) \cap Y = \emptyset$. Take $B(t_1)$ to be maximal. Then $\bigcup_{t_1 \in Y}B(t_1)$ is denoted by $T_G(Y)$, or $T(Y)$. 
\end{defn}

\begin{defn}
For a graph $G$, a set $Y \subseteq V(G)$, and a vertex $v_1 \in T_G(Y)$, let $t_1 \in N_G(v_1) \cap Y$. Delete the edge $v_1t_1$. For each $t_2 \in N_G(v_1) \setminus \{t_1\}$, subdivide the edge $v_1t_2$ by a new vertex $w_2$ respectively. For the set of all $v_1$ applied this replacement, say $S$, the resulting graph is denoted by $G(S)$.
\end{defn}

\begin{fact}\label{T'0}
For a graph $G$, a set $Y \subseteq V(G)$, and a set $T' \subseteq T_G(Y)$, if $Y$ is a d-set of $G$, then $Y \cup T'$ is a dominating set of $G(T')$. 
\end{fact}

\begin{fact}\label{T'}
For a graph $G$, a set $Y \subseteq V(G)$, and a set $T' \subseteq T_G(Y)$, if $Y$ is a d-set of $G - T'$, then $Y \cup T'$ is a d-set of $G(T')$. 
\end{fact}

\begin{thm}\label{T1}
For a connected cubic graph $G$, if $\gamma(G) > \lceil |V(G)|/3 \rceil$, then $\gamma(G) = i(G)$.
\end{thm} 

\begin{proof}
Let $G$ be a connected cubic graph and $X$ be its d-set. Suppose that $E(X)$ is minimal. Suppose $||X|| > 0$. Now, there exists a path $P$ in $G$ such that $P = a_1a_2$ for some $a_1, a_2 \in X$. Now we construct a graph $G''$ from $G$. Let $G_0 = G$. Let $i$ be an integer such that $i \geq 0$. For some $e \in U_{G_i}(X)$, let $G_{i + 1} = G_i - e$. For some $j$ such that $j \geq i$, let $G' = G_{j + 1}$. By Fact \ref{U'}, $X$ is a d-set of $G'$. For some $T' \subseteq T_{G'}(X)$, let $G'' = G'(T')$. Let $Y = X \cup T'$. By Fact \ref{T'0}, $Y$ is a dominating set of $G''$. Let $X'$ be a d-set of $G''$. Let $A$ be a path (or a cycle) component of $G''$ such that $N_A[Y(A)] = V(A)$. Let $\mathcal{A}$ be the set of all $A$. 

\begin{claim}\label{gaaa}
It is possible to take $G''$ as $G'' = \bigcup_{A \in \mathcal{A}} A$.
\end{claim}

\begin{proof}
Let $G_0 = G$. Let $i$ be an integer such that $i \geq 0$. For some $v \in V(G_i)$, suppose $N_{G_i}(v) = \{w_1, w_2, w_3\}$. If $vw_1 \in U_{G_i}(X)$, then by Fact \ref{U'}, $X$ is a d-set of $G_i - vw_1$. Let $G_{i + 1} = G_i - vw_1$. Take $j \geq i$ to be maximal. Let $G' = G_{j + 1}$. For some $x \in V(G') \setminus X$, suppose $|N_{G'}(x)| = 3$. Since $X$ is a dominating set, we have $N_{G'}(x) \cap X \ne \emptyset$.  If $|N_{G'}(x) \cap X| \geq 2$, it contradicts the definition of $G'$. Thus $|N_{G'}(x) \cap X| = 1$. Now, $x \in T_{G'}(X)$. Let $T_1$ be the set of all $x$. For some $v \in X$, suppose $N_{G'}(v) = \{w_1, w_2, w_3\}$. By the definition of $G'$, we have $N_{G'}(v) \subseteq T_{G'}(X)$. If $|N_{G'}(w_1)| = 3$, let $y = w_1$. Otherwise, if $|N_{G'}(w_2)| = 3$, let $y = w_2$. Otherwise, if $|N_{G'}(w_3)| = 3$, let $y = w_3$. Otherwise, let $y = w_1$. Let $T_2$ be the set of all $y$. Let $T' = T_1 \cup T_2$. Let $G'' = G'(T')$. Therefore, $G''$ is the union of path (or cycle) components. Let $A$ be a path (or a cycle) component of $G''$. Let $Y = X \cup T'$. By Fact \ref{T'0}, $Y$ is a dominating set of $G''$, and so $N_A[Y(A)] = V(A)$. 
\end{proof}

That is, $|Y| - |X'| = \Sigma_{A \in \mathcal{A}} (|Y(A)| - |X'(A)|)$. In addition, take $G''$ as $|Y| - |X'|$ is maximum.
We prove our theorem by induction for $|Y| - |X'|$. First, let $|Y| - |X'| = 0$. Suppose that there exists a path $Q$ in $G$ such that $Q = b_1b_2b_3$ for some $b_1, b_3 \in X$ and some $b_2 \not \in X$. By Fact \ref{T'}, $Y$ is a d-set of $G''$. Suppose $P \cdots Q \subseteq G''$. Let $R = P \cdots Q$. Even for a path $R' = \alpha R \beta$ such that $\alpha, \beta \notin Y$, $Y(R')$ is not a d-set of $R'$, also for $R$, a contradiction. Thus $R \not \subseteq G''$. If $Q \subseteq G$, then $P \cdots b_2b_1 \subseteq G''$ or $P \cdots b_2b_3 \subseteq G''$, where $|N_G(b_1) \setminus N_G[X \setminus \{b_1\}]| = 2$ and $|N_G(b_3) \setminus N_G[X \setminus \{b_3\}]| = 2$. After all, by Lemma \ref{disjoint}, we have $|X| \leq \lceil |V(G)|/3 \rceil$. 

Suppose that if $|Y| - |X'| \leq k$ ($k \geq 0$), then $|X| \leq \lceil |V(G)|/3 \rceil$. Let $|Y| - |X'| = k + 1$. Suppose that there exists a path $Q$ in $G$ such that $Q = b_1b_2b_3$ for some $b_1, b_3 \in X$ and some $b_2 \not \in X$, and $P \cdots Q \subseteq A \in \mathcal{A}$, for otherwise, by Lemma \ref{disjoint}, we have $|X| \leq \lceil |V(G)|/3 \rceil$. Let $H_0 = A$ and $Y_0 = Y(A)$.
Let $i$ be an integer such that $0 \leq i \leq j - 1$ ($j \geq 1$). Take a vertex $v_i \in Y_i$ such that $N_{H_i}[v_i] \setminus \{v_i\} \subseteq V(H_i) \setminus Y_i$. Let $S(v_i) = N_{H_i}[v_i]$. Note that $|S(v_i)| = 3$. Let $e_i$ be a new edge between two vertices of $N_{H_i}(S(v_i)) \setminus S(v_i)$ if there exist. Let $H_{i + 1} = H_i - S(v_i) + e_i$. Let $Y_{i + 1}$ be constructed from $Y_i$ as follows; delete $v_i$, and for some $x_1, x_2 \in Y_i$, if $x_1x_2 \in E(H_{i + 1})$, then delete $x_2$ and add the vertex of $N_{H_{i + 1}}(x_2) \setminus \{x_1\}$ in order to take a vertex $v_{i + 1} \in Y_{i + 1}$ in the next step. Let $B = H_j$. Let $H''$ be a graph constructed from $G''$ by deleting $A$ and adding $B$. Let $O = (Y \setminus Y(A)) \cup Y_j$. Let g$Z'$ be a d-set of $H''$. Now, $|O(B)| - |Z'(B)| = |Y(A)| - |X'(A)| - 1$ for some $j$ such that $j \geq 1$. Thus $|O| - |Z'| \leq k$. By the induction hypothesis, $|X| - j \leq \lceil (|V(G)| - 3j)/3 \rceil$, which implies $|X| \leq \lceil |V(G)|/3 \rceil$. Therefore, the proof of Theorem \ref{T1} is complete. 
\end{proof}

\end{document}